\documentclass[12 pt]{amsart}

\usepackage{amsmath, amssymb, mathtools, palatino, euler, epic, eepic, xypic, floatflt, microtype}
\usepackage{graphicx}
\usepackage[usenames]{xcolor}
\usepackage{float}
\usepackage{tikz}
\usepackage{subcaption}

\usepackage{fullpage}
\usepackage{amscd,amsthm}
\usepackage{comment}
\usepackage{color}

\usepackage[margin=1 in]{geometry}
\definecolor{darkblue}{rgb}{0,0,0.7}
\definecolor{darkred}{rgb}{0.7,0,0}
\usepackage[colorlinks,linkcolor=darkred, citecolor=darkblue,
pagebackref=true,pdftex]{hyperref}

\DeclareMathOperator\dual{Dual}
\DeclareMathOperator\init{in}
\DeclareMathOperator\fin{fin}

\newcommand\C{\mathcal C}
\newcommand\D{\mathcal D}
\DeclareMathOperator{\lk}{link}
\DeclareMathOperator{\del}{del}
\DeclareMathOperator\ind{Ind}

\newcommand\set[1]{\left\{ #1 \right\}}
\newcommand\abs[1]{\left| #1 \right|}
\newcommand\parens[1]{\left( #1 \right)}

\newcommand\ideal[1]{\langle #1 \rangle}

    \newtheorem{thm}{Theorem}[section]
    \newtheorem{prop}[thm]{Proposition}
    \newtheorem{lem}[thm]{Lemma}
    
    \newtheorem{cor}[thm]{Corollary}
    \newtheorem{defn}[thm]{Definition}
    \newtheorem{rem}[thm]{Remark}
    \newtheorem{example}[thm]{Example}
    \newtheorem{question}[thm]{Question}

\definecolor{FoundersBlue}{RGB}{0,59,112} 
\definecolor{AntonMaroon}{RGB}{80, 18, 20} 
\definecolor{AntonRojo}{RGB}{227, 40, 73} 

\begin{document}

\title[Chordality, syzygies, and shellability for hypergraphic analogues of interval graphs]{Chordality, syzygies, and shellability \\ for hypergraphic analogues of interval graphs}

\author[A. Dochtermann]{Anton Dochtermann}
\address{Department of Mathematics, Texas State University, San Marcos, TX 78666, USA}
\email{dochtermann@txstate.edu}

\author[B. Goeckner]{Bennet Goeckner}
\address{Department of Mathematics, University of San Diego, San Diego, CA 92110, USA}
\email{bgoeckner@sandiego.edu}

\author[M. Pavelka]{Marta Pavelka}
\address{QMATH, University of Copenhagen, Universitetsparken 5, DK-2100 Copenhagen, Denmark}
\email{mp@math.ku.dk}


\begin{abstract}

 Interval graphs are a special class of chordal graphs, and hence have connections to commutative algebra via Fr\"oberg's theorem that characterizes linear resolutions of squarefree quadratic ideals. In recent years, several hypergraphic analogues of interval and chordal graphs have been proposed, in part as an effort to extend Fr\"oberg's theorem to ideals generated in higher degree. In this paper, we study two such classes from the literature, cointerval hypergraphs and underclosed complexes, and show that they are in fact equivalent up to complementation. We then consider their place in the broader theory of higher-dimensional chordality, proving that an underclosed clutter is chordal in the sense of Woodroofe. As a consequence, we answer a question of Dochtermann and Engstr\"om by showing that the associated Alexander dual complexes are vertex decomposable, implying that the corresponding circuit ideals have linear quotients. We furthermore show that these dual complexes have shellings induced by their underclosed vertex orders.
    \end{abstract}
    
\maketitle

\section{Introduction}

 Many applications of graph theory involve understanding how properties defining various classes of graphs translate to the setting of algebra and topology. In recent years, some of these classes of graphs have been generalized to hypergraphs (where edges can have arbitrary cardinality), with the goal of extending these connections to higher-dimensional (or often higher degree) settings. 

For instance, the class of \emph{chordal} graphs have a close connection to commutative algebra: a result of Fr\"oberg says that a graph $G$ is the complement of a chordal graph if and only if the edge ideal $I(G)$ has a linear resolution (see below for definitions).
A substantial body of work has been devoted to finding a `higher-dimensional' version of Fr\"oberg's theorem, i.e. to find a combinatorial characterization (or at least necessary or sufficient conditions) of $k$-uniform \emph{hypergraphs} whose edge ideals (in this context often called \emph{circuit ideals}) have a linear resolution. Note that the `circuit ideal of a $k$-uniform hypergraph' is just another name for a square-free monomial ideal generated in degree $k$. 

In recent years, a number of authors have introduced various candidates for the notion of a chordal hypergraph. The general idea is to use one of the many characterizations of chordal graphs and consider higher-dimensional analogues. For instance, the property of every $n$-cycle `having a chord' 
is generalized in \cite{ConFar}, the existence of a \emph{perfect elimination order} and the related notions of a \emph{simplicial} vertex/ridge is studied in different ways for higher-dimensional complexes and hypergraphs in \cite{StabilityChordal}, \cite{EmtanderChordal}, and \cite{Russ09}, the characterization of chordal graphs in terms of removing \emph{exposed} edges is pursued in \cite{DocExposed}, and the \emph{Leray property} of clique complexes of chordal graphs is generalized to simplicial complexes in \cite{AdiNevSam}.

In \cite{DocEng}, the first author and Engstr\"om seek to describe (minimal) \emph{cellular} resolutions of hypergraph circuit ideals that have linear resolutions; that is, to construct polyhedral complexes whose face structure encodes the (minimal) syzygies of such ideals. They show that if $\C$ is a so-called \emph{cointerval} $k$-hypergraph, then the circuit ideal $I(\C)$  has a minimal linear resolution supported on a polyhedral complex that parametrizes certain hypergraph \emph{homomorphisms}.

Cointerval $k$-hypergraphs have a recursive definition (see below), and in the case $k=2$, it recovers the set of graphs that are complements of interval graphs (hence the name). Interval graphs, defined as the intersection graphs of intervals in the real line, are a well studied class of graphs that form a proper subclass of chordal graphs. Thus if one is seeking a higher-dimensional analogue of Fr\"oberg's theorem, one should view a cointerval $k$-hypergraph as an example of a `complement of a chordal hypergraph'. 

Interval graphs can also be characterized by the existence of an ordering of the underlying vertex set that satisfies a certain closure property.  In \cite{BenSecVar}, Benedetti, Seccia, and Varbaro use this description to define the notion of an \emph{underclosed} $k$-hypergraph (so that in particular, an underclosed $2$-hypergraph is an interval graph). The authors of \cite{BenSecVar} work with pure simplicial complexes in their definition, but for our purposes the hypergraph/clutter perspective will be more convenient.
The motivation in \cite{BenSecVar} comes from extending certain results from the theory of binomial edge ideals (associated with a graph $G$) to the setting of determinantal facet ideals (defined by a uniform hypergraph).

\subsection{Our contributions}
In this work, we study combinatorial and algebraic properties of complexes and ideals defined by cointerval and underclosed hypergraphs, and also seek to understand how the different classes relate to each other. 

 As mentioned above, we work in the setting of hypergraphs, where by definition a hypergraph $\C$ is a collection of subsets (called circuits) of some ground set $V = V(\C)$. We will always assume that our hypergraphs are in fact \emph{clutters}, meaning that no circuit contains another. When all the circuits of $\C$ have the same cardinality (say $k$), we say that $\C$ is a uniform $k$-hypergraph (or simply $k$-hypergraph). Note that any uniform $k$-hypergraph is a clutter, in this context, we will use both $k$-hypergraph and $k$-clutter.

Our first result demonstrates a close connection between the cointerval $k$-hypergraphs from \cite{DocEng} and the underclosed complexes from \cite{BenSecVar}. In what follows, if ${\mathcal C}$ is a $k$-clutter with vertex set $V$ then its complement $\overline{\mathcal C}$ is the clutter on the same vertex set whose circuits are the $k$-subsets of $V$ that are not circuits of ${\mathcal C}$.

\newtheorem*{thm:complement}{Theorem \ref{thm:complement}}
\begin{thm:complement}
A $k$-clutter ${\mathcal C}$ is cointerval if and only if its complement $\overline{\mathcal C}$ is underclosed.
\end{thm:complement}

We refer to Section \ref{sec:background} for relevant definitions. We next investigate how underclosed clutters (which we now know include complements of cointerval $k$-hypergraphs) are related to other notions of chordality from the literature.

As pointed out in \cite[Remark~36]{BenSecVar}, the class of underclosed $k$-clutters does not sit inside the class of chordal clutters as defined by Emtander in \cite{EmtanderChordal}. However, we are able to relate underclosed complexes to a class introduced by Woodroofe in \cite{Russ09}, which we will refer to as `chordal clutters'.

\newtheorem*{thm:underclosedchordal}{Theorem \ref{thm:underclosedchordal}}
\begin{thm:underclosedchordal}
Suppose ${\mathcal C}$ is an underclosed (not necessarily uniform) clutter. Then ${\mathcal C}$ is chordal. 
\end{thm:underclosedchordal}

 This leads to a number of corollaries and applications, and in particular answers a question from \cite{DocEng} that we sketch here. In \cite{DocEng} the authors show that the circuit ideal $I({\mathcal C})$ of any cointerval $k$-hypergraph ${\mathcal C}$ has a linear resolution. A stronger condition, known as `having linear quotients', would follow from a topological property of a related simplicial complex. In particular, since $I({\mathcal C})$ is a squarefree monomial ideal, it can be recovered as the Stanley--Reisner ideal $I({\mathcal C}) = I_\Delta$ of some simplicial complex $\Delta$ (see Section \ref{sec:commalg}). 
In \cite{DocEng}, the authors ask whether the Alexander dual of this related complex $\Delta$ (which we denote by $\dual(\C)$ here) is shellable --- this would imply that the ideal $I({\mathcal C})$ has linear quotients. As a corollary of Theorem \ref{thm:underclosedchordal} we obtain a positive answer to this question.

\newtheorem*{cor:dualisVD}{Corollary \ref{cor:dualisVD}}
\begin{cor:dualisVD}
Suppose ${\mathcal C}$ is an underclosed $k$-clutter. Then $\dual(\C)$ is vertex decomposable and thus shellable.
\end{cor:dualisVD}

Our last result addresses further structure of these dual complexes. By definition, a shelling of a simplicial complex $\Delta$ involves ordering its facets. Since each facet is a subset of the vertex set $V$, a natural question to ask is whether such an ordering can be obtained as the lexicographic order induced by some linear ordering of $V$. Complexes with this lexicographic (or lex) shelling properties arise naturally (in fact can be used to characterize the independence complexes of matroids). Unlike the relationship between vertex decomposability and shellability in general, it is shown in \cite{GP24} that there exist vertex decomposable complexes that are not lex shellable. We show that our dual complexes do in fact have this lex shelling property.

\newtheorem*{thm:DualIsLexShell}{Theorem \ref{thm:DualIsLexShell}}
\begin{thm:DualIsLexShell}

Suppose $\C$ is a uniform underclosed clutter with respect to the labeling of the vertex set. Then $\dual(\C)$ is lex shellable with respect to this ordering. 
\end{thm:DualIsLexShell}

\subsection{Organization}
The paper is organized as follows. In Section \ref{sec:background} we provide definitions and notation, including basic properties of simplicial complexes and relevant notions from commutative algebra. In Section \ref{sec:underclosedcoint} we remind the definitions of underclosed and cointerval clutters, and relate these classes of clutters in Theorem \ref{thm:complement}. In Section \ref{sec:chordal} we recall the notion of chordal clutters from \cite{Russ09} and prove Theorem \ref{thm:underclosedchordal}, showing that underclosed clutters are chordal. In Section \ref{sec:shell} we discuss the notion of lexicographic (lex) shellings and prove Theorem \ref{thm:DualIsLexShell}, which shows that duals of underclosed clutters have lex shellings. In Section \ref{sec:further} we present some open questions and discuss possible avenues for further research.

\section{Background} \label{sec:background}
 We start with some basic definitions that will be used throughout the paper.
A \emph{simplicial complex} $\Delta$ on vertex $V = V(\Delta)$ is a collection of subsets of $V$ (called \emph{faces} or \emph{simplices}) such that if $\sigma$ is a face of $\Delta$ and $\tau \subseteq \sigma$, then $\tau$ is also a face of $\Delta$. The singleton faces $\{v\}$ of $\Delta$ are called \emph{vertices} of $\Delta$, but we do not require that all elements of $V$ are faces of $\Delta$.  
A \emph{facet} of $\Delta$ is a face that is maximal (under inclusion of sets). The \emph{dimension} of a face $\sigma \in \Delta$ is $\dim(\sigma)= \abs{\sigma}-1$ and the \emph{dimension} of $\Delta$ is $\dim(\Delta) = \max \set{ \dim (\sigma) : \sigma \in \Delta}.$  We say that $\Delta$ is \emph{pure} if all its facets have the same cardinality. Facets $F$ and $G$ of a pure $(d-1)$-dimensional complex are \emph{adjacent} if they intersect in a codimension one face, i.e. $|F \cap G| = d-1$.

A \emph{hypergraph} ${\mathcal C}$ on vertex set $V = V({\mathcal C})$ is a collection of subsets $E = E(\C)$  of $V$ (called \emph{circuits} or \emph{edges}). We say that ${\mathcal C}$ is a \emph{clutter} if no circuit is properly contained in another. A \emph{$k$-circuit} is a circuit of ${\mathcal C}$ that has $k$-elements, and ${\mathcal C}$ is \emph{$k$-uniform} if every circuit of ${\mathcal C}$ is a $k$-circuit. In this case, we will say that ${\mathcal C}$ is a \emph{$k$-clutter}. Note that all $k$-uniform hypergraphs are $k$-clutters.  If $\C$ is a $k$-clutter on vertex set $V$, the \emph{complement} $\overline{\C}$ is the $k$-clutter whose circuits are given by all $k$-subsets of $V$ that are not circuits of $\C$.

For a clutter ${\mathcal C}$, its \emph{independence complex} $\ind({\mathcal C})$ is the simplicial complex on vertex set $V$ whose faces are given by all subsets of $V$ that do not contain a circuit.  On the other hand, for a simplicial complex $\Delta$, its \emph{minimal nonface clutter} ${\mathcal C}(\Delta)$ consists of all minimal nonfaces of $\Delta$. It is straightforward to verify that ${\mathcal C}(\ind({\mathcal C})) = {\mathcal C}$ and $\ind({\mathcal C}(\Delta)) = \Delta$. 

All simplicial complexes and clutters considered here have finite vertex sets. For a positive integer $n$, we use the convention $[n]=\set{1,2,\dots, n}.$

\begin{rem}
In what follows, we will often assume that the vertex set of our simplicial complex (or clutter) is given by $V = [n]$ (or more generally $V \subseteq {\mathbb Z}$). In this case, unless otherwise specified, we write our faces (or circuits) in increasing order $F = a_1a_2 \cdots a_k$, where $a_1 < a_2 < \cdots < a_k$.
\end{rem}

\subsection{Shellability and vertex decomposability}\label{sec:shell}

We next recall some combinatorial and topological concepts related to simplicial complexes. Recall that a simplicial complex $\Delta$ is \emph{shellable} if there exists an ordering of its facets $F_1, F_2, \dots, F_m$ such that for any $i = 2, \dots m$ the intersection of $F_i$ with the complex $\langle F_1, \dots, F_{i-1} \rangle$ generated by the previous facets is pure of dimension $\dim F_i - 1$.  We will typically work in the pure setting, where $\Delta$ is $(d-1)$-dimensional, and the intersection is $(d-2)$-dimensional.  We refer to the addition of each $F_i$ in a shelling as a \emph{shelling move}.

Provan and Billera \cite{PB80} introduced a related class of complexes. Suppose $\Delta$ is a simplicial complex on vertex set $V$. Given a face $\sigma \in \Delta$ we define its \emph{deletion} $\del_\Delta(\sigma)$ and \emph{link} $\lk_\Delta(\sigma)$ as follows:
\begin{align*}
    \del_\Delta(\sigma) &= \set{ \tau \in \Delta ~:~ \sigma \not \subseteq \tau} \\
    \lk_\Delta(\sigma)  &= \set{ \tau \in \Delta ~:~ \tau \cup \sigma \in \Delta, \;  \tau \cap \sigma = \varnothing}
\end{align*}
A vertex $v \in V$ is a \emph{shedding vertex} if every facet of the deletion $\del_\Delta(v)$ is a facet of $\Delta$.
A simplicial complex $\Delta$ is \emph{vertex decomposable} if $\Delta$ is either a simplex or else admits a shedding vertex $v$ such that both $\lk_\Delta(v)$ and $\del_\Delta(v)$ are vertex decomposable. It is known \cite[Corollary 2.9]{PB80} that if $\Delta$ is vertex decomposable then $\Delta$ is also shellable.

\subsection{Some commutative algebra}\label{sec:commalg}

Many of the above constructions can be interpreted in the context of commutative algebra. Indeed, the motivation for our definitions comes from open questions in this setting. We review the basic ideas here and refer to \cite{CombCommAlg} for more details.

If $\Delta$ is a simplicial complex on vertex set $V = [n]$, we fix a field ${\mathbb K}$ and let $S = {\mathbb K}[x_1, \dots, x_n]$ denote the polynomial ring with indeterminates indexed by the vertices. The \emph{Stanley--Reisner ideal} $I_\Delta$ is the ideal in $S$ generated by monomials corresponding to nonfaces of~$\Delta$:
\[I_\Delta = \langle x_{i_1}x_{i_2} \cdots x_{i_k} : i_1i_2 \cdots i_k \notin \Delta \rangle. \]
The \emph{Stanley--Reisner ring} ${\mathbb K}[\Delta]$ is then defined as the quotient ring ${\mathbb K}[\Delta] = S/I_\Delta$.

On the other hand, if ${\mathcal C}$ is a clutter on vertex set $[n]$, we defined the \emph{edge ideal} (or sometimes \emph{circuit ideal}) $I({\C})$ to be the monomial ideal generated by the circuits of $\C$:
\[I({\C}) = \langle x_{i_1}x_{i_2} \cdots x_{i_k} : i_1i_2 \cdots i_k \in E(\C) \rangle.\]
Using the language introduced above, one can see \cite{Russ09} that for any clutter $\C$ we have $I(\C) = I_{\ind(\C)}$ and for any simplicial complex $\Delta$ we have $I_\Delta = I(\C(\Delta))$.

Suppose $I$ is any ideal in the polynomial ring $S$. We say that $I$ has a \emph{linear resolution} if the $S$-module $S/I$ has a minimal resolution where all nonzero entries in some choice of matrices for the differential maps have linear entries. If $I$ is generated in fixed degree $d$, this property also has an interpretation in terms of the Betti numbers of $I$, see \cite{CombCommAlg}. 

If $G$ is a graph (a $2$-clutter) on vertex set $V = [n]$, then $I(G)$ is a squarefree quadratic monomial ideal. A well-known result of Fr\"oberg \cite{Froberg} says that $I(G)$ has a linear resolution if and only if $\overline G$ is a chordal graph. Recall that $\overline G$ denotes the \emph{complement} $2$-clutter, with circuits (edges) given by all $2$-subsets of $V$ that are not circuits of $G$. 

A useful result of Eagon and Reiner \cite{EagonReiner} tells us that $I$ has a linear resolution if and only if its Alexander dual $I^\vee$ (see, e.g., \cite[Chapter~5]{CombCommAlg}) is Cohen--Macaulay over ${\mathbb K}$.  
For the case where $I = I_\Delta$ is the Stanley--Reisner ideal of a simplicial complex $\Delta$, this follows (for any field ${\mathbb K}$) if $\Delta$ is shellable.
Recall that if $\Delta$ is a simplicial complex on vertex set $V = [n]$, the Alexander dual $\Delta^\vee$ is the simplicial complex whose faces are complements of nonfaces of $\Delta$, so that
\[F \in \Delta^\vee \Leftrightarrow [n] \setminus F \notin \Delta.\]
If $I = I_\Delta$ is a Stanley--Reisner ideal, one can recover the Alexander dual ideal as $I^\vee = I_{\Delta^\vee}$,
the Stanley--Reisner ideal of this dual complex.

 In the case where $I = I(\C)$ is the circuit ideal of a $k$-clutter $\C$, the dual complex has a straightforward construction that we record as a definition.

\begin{defn}\label{defn:dual}
Suppose $\C$ is a $k$-clutter on vertex set $V = [n]$. Define the \emph{dual complex} of $\C$, denoted $\dual(\C)$, to be the simplicial complex on $V$, whose facets are given by the sets 
\[\{[n] \setminus e: e \notin E(\C)\}.\]
\end{defn}

Hence $\dual(\C)$ is a pure $(n-k-1)$-dimensional simplicial complex, and the facets of $\dual(\C)$ are complements of circuits of $\overline{\C}.$ Furthermore, its Stanley--Reisner ideal $I_{\dual(\C)}$ is the Alexander dual of the circuit ideal $I(\overline{\C})$. 
We refer to Example~\ref{ex:dual} for an illustration. 

\begin{example} \label{ex:dual}
Suppose that $\C$ is the $2$-clutter (graph) on vertex set $V = [5]$ with edge set $E(\C)=\{12, 13, 23, 34, 45\}$. We see that $\C$ is a chordal graph and that its complement $\overline{\C}$ has edge set $E(\overline{\C})=\set{14,15,24,25,35}$. This implies that the circuit ideal of the complement  $I(\overline{\C}) = \langle x_1x_4, x_1x_5, x_2x_4, x_2x_5, x_3x_5 \rangle$ has a linear resolution. The dual complex $\dual(\C)$ is the pure $2$-dimensional simplicial complex on $[5]$ with facets given by $\{124, 134, 135, 234, 235 \}.$  One can also check that in this case $\dual(\C)$ is shellable, using the given lexicographic order on the facets.
\end{example}

\section{Underclosed and cointerval clutters}\label{sec:underclosedcoint}

In this section, we consider cointerval and underclosed clutters and discover a close connection between these classes.

We now recall the definition of cointerval $k$-clutters, first studied in \cite{DocEng} (where they were called $k$-hypergraphs). In what follows, if $\C$ is a $k$-clutter on vertex set $V \subseteq {\mathbb Z}$, the \emph{$i$-layer} of ${\mathcal C}$ is the $(k-1)$-clutter $\C^i$ on vertex set $V \setminus \{1, 2, \dots, i\}$ with circuits given by
\[E(\C^i)= \{a_2a_3 \cdots a_k : i a_2 a_3 \cdots a_k \in E({\mathcal C})\}.\]
\noindent
 Recalling our convention that a circuit $E = ia_2a_3 \cdots a_k$ is written so that $i < a_2 < a_3 < \cdots < a_k$, the $i$-layer of $\C$ is the collection of all circuits that contain $i$ as the lowest element, with the element $i$ removed.

\begin{defn}[Cointerval clutter]\label{defn:coint}  \cite[Definition 4.1]{DocEng}
The class of \emph{cointerval} (uniform) $k$-clutters on vertex set $V \subseteq {\mathbb Z}$ is defined iteratively as follows. All $1$-clutters are cointerval, and for $k > 1$, a clutter ${\mathcal C}$ is cointerval if
\begin{enumerate}
\item
For every $i \in V$ the clutter $\C^i$ is cointerval;
\item
For all $i < j$, we have that $\C^j$ is a subclutter of $\C^i$. 
\end{enumerate}
\end{defn}

As shown in \cite{DocEng}, one can check that a $2$-clutter is cointerval if and only if it is the complement of an interval graph.

Another approach to generalizing interval graphs is proposed in \cite{BenSecVar}. Although the authors phrase their construction in terms of simplicial complexes, we use the language of clutters.

\begin{defn}[Underclosed clutter] \label{defn:underclosed}
Suppose ${\mathcal C}$ is a clutter with $n$ vertices. Then ${\mathcal C}$ is \emph{underclosed} if there exists a (re)ordering of its vertex set $V = \{1, 2, \dots, n\}$ such that if $e = a_1 a_2 \cdots a_k$ is a $k$-circuit of ${\mathcal C}$ then $a_1b_2b_3 \cdots b_k$ is also a $k$-circuit for any $b_2 \leq a_2, b_3 \leq a_3, \dots, b_k \leq a_k$. 
\end{defn}

Note that we do not require ${\mathcal C}$ to be a uniform clutter.
One can see that an interval graph $G$ (represented as a collection of $n$ intervals in ${\mathbb R}$) defines an underclosed $2$-clutter by ordering the intervals according to their left endpoint (see for instance \cite{Wood06}). From the definition, it is also clear that an underclosed clutter is a generalization of a \emph{shifted} set system \cite{Frankl87}. 
In \cite{Dinginterval}, Ding defined a clutter to be \emph{interval} if its vertices can be linearly ordered so that every circuit is a consecutive set of vertices. It is clear that an underclosed clutter is interval in this sense. 
See Section \ref{sec:further} for more discussion regarding connections to other classes of clutters.

\begin{lem}\label{lem:smallest}
If ${\mathcal C}$ is an underclosed clutter with respect to some ordering $\{1,2, \dots, n\}$ of the vertex set, then the vertex $1$ cannot be in circuits of different cardinalities.
\end{lem}

\begin{proof}
Suppose $e_1$ and $e_2$ are circuits of ${\mathcal C}$. Assume $1 \in e_1$ and $1 \in e_2$, where $|e_1| = k_1$, $|e_2| = k_2$, and $k_1 < k_2$. Then the underclosed property implies that ${\mathcal C}$ contains circuit $12\cdots k_1$ as well as $12 \cdots k_2$, a contradiction to the clutter property.
\end{proof}

By the same argument as above, one can see that any vertex cannot be the minimum element in circuits of different sizes in an underclosed clutter.

Recall that for the case of $k=2$, by construction the class of cointerval $k$-clutters coincides with the complements of underclosed $k$-clutters. In this section, we show that this holds for all $k$. We begin with the following observation regarding cointerval clutters.

\begin{lem} \label{lem:coint}
Suppose ${\mathcal C}$ is a cointerval $k$-clutter, for $k \geq 1$. If $e = b_1 b_2 \cdots b_k \in E({\mathcal C}) $, then $a_1a_2 \cdots a_{k-1}b_k \in E({\mathcal C})$, for any $a_1 \leq b_1, a_2 \leq b_2, \dots, a_{k-1} \leq b_{k-1}$. 
\end{lem}

\begin{proof}
Note that if $k = 1$ then the condition is clearly satisfied, so we assume $k \geq 2$.

Suppose ${\mathcal C}$ is a $k$-clutter containing the circuit $e = b_1 b_2 \cdots b_k$. Note if $j_1 < b_1$ then $j_1b_2 \cdots b_k$ is in $E({\mathcal C})$ by condition $(2)$ of Definition \ref{defn:coint}. 

Hence $e^1 = b_2 \dots b_k$ is a circuit of $\C^{j_1}$, the $j_1$-layer of $\C$. Recall that $\C^{j_1}$ is a cointerval $(k-1)$-clutter by condition $(1)$ of Definition \ref{defn:coint}.  By induction on $k$ we have that $j_2 \cdots j_{k-1}b_k \in E(\C^{j_1})$. Hence we have $j_1j_2 \cdots j_{k-1}b_k \in E(\C)$, as desired.
\end{proof}

In what follows, recall that if $\C$ is a uniform $k$-clutter on vertex set $V$, its {complement} $\overline{\C}$ is the $k$-clutter whose circuits are all $k$-subsets of $V$ that do not belong to $\C$. Further, if $1<2<\dots<n-1<n$ is an order, we say that its \emph{reverse order} is $n <_{r} n-1 <_{r} \dots <_{r} 2 <_{r} 1.$

\begin{thm}\label{thm:complement}
A uniform clutter ${\mathcal C}$ on vertex set $V \subseteq {\mathbb Z}$ is cointerval if and only if its complement $\overline{\C}$ is underclosed with respect to the reverse ordering.
\end{thm}

\begin{proof}
First we assume that ${\mathcal C}$ is a cointerval $k$-clutter.
Suppose we have $e = a_1a_2 \cdots a_k \in \overline{\C}$, where $a_1 <_{r} a_2 <_{r} \cdots <_{r} a_k$. Note that we are using the reverse order when describing circuits in $\overline{\C}$ (so that in fact $a_k < a_{k-1} < \cdots < a_1$).

Let $f=a_1 b_2 \dots b_k$ where $b_i \le_{r} a_i$ for all $i \in \{2, \dots, k\}$. We wish to show that $f \in \overline{C}.$ If not, we would have $f \in \C$, i.e., $b_k \dots b_2 a_1 \in \C.$ By Lemma~\ref{lem:coint}, this shows that $a_k \dots a_2 a_1 \in \C$, which is a contradiction. We conclude that $\overline{\mathcal C}$ is underclosed.

Now suppose $\C$ is a $k$-clutter such that its complement $\overline{\C}$ is underclosed on $V = \{n, n-1, \dots, 1\}$ under the reverse ordering $n <_{r} n-1 <_{r} \dots <_{r} 1$ as described above. We claim that $\C$ is cointerval under the usual ordering of $V$. We prove this by induction on $k$.  

Note that for $k=1$, all $1$-clutters (i.e. singleton subsets of $V$) are both underclosed and cointerval regardless of the underlying ordering of $V$.

 For $k > 1$, we start by checking the first cointerval condition, namely that the $i$-layer of $\C$ is cointerval (under the usual order). By induction it is enough to show that $\overline{\C^i}$, the complement of the $i$-layer, is underclosed.

 Note that $\overline{\C^i}$ consists of all $(k-1)$-subsets of the form $b_1 b_2 \cdots b_{k-1}$  such that $b_1 b_2 \cdots b_{k-1}$ is not a circuit of $\C^i$. In particular $i b_1 b_2 \cdots b_{k-1}$ is not a circuit of $\C$, so that $b_{k-1} b_{k-2} \cdots b_1 i$ (in the reverse order) is a circuit of $\overline{\C}$.
 See Example \ref{ex:undercoint} for an illustration.

Now suppose $b_{k-1} b_{k-2} \cdots b_1$ is in $\overline{\C^i}$, and let $b_{k-1} <_{r} j_{k-2} <_{r} j_{k-3} <_{r} \cdots <_{r} j_1$ with 
$j_{\ell} \leq_{r} b_\ell$ for all $\ell \in \{1, \dots, k-2\}$.
We then have that $b_{k-1} j_{k-2} j_{k-3} \cdots j_1 i$ is a circuit of $\overline{\C}$, by the underclosed property of $\overline{\C}$. Hence $b_{k-1} j_{k-2} j_{k-3} \cdots j_1$ is in  $\overline{\C^i}$. We conclude that $\overline{\C^i}$ is underclosed, and hence by induction the $i$-layer is cointerval.

For the second condition, suppose $i < j$. We must show that $\C^j$, the $j$-layer of $\C$, is a subclutter of $\C^i$, the $i$-layer in $\C$. Suppose $b_1 b_2 \cdots b_{k-1}$ is $\C^j$, so that $j b_1 b_2 \cdots b_{k-1}$ is a circuit of $\C$. We claim that $i b_1 b_2 \cdots b_{k-1}$ is a circuit of $\C$. If not, then (after reversing order), we would have $b_{k-1} b_{k-2} \cdots b_1 i$ is a circuit of $\overline{\C}$. But in this ordering we have $j <_{r} i$, and hence by the underclosed property we have $b_{k-1} b_{k-2} \cdots b_1 j \in \overline{\C}$, a contradiction.
\end{proof}

 \begin{example}\label{ex:undercoint} 
Let $\C$ be a $4$-clutter on $[6]$ with the standard order on $[6]$ and circuits
\begin{equation*}
E( \C) = \set{1234, 1235, 1236, 1246, 1256, 1346, 1356, 1456}.
\end{equation*} 
Its complement $\overline{\C}$ is a $4$-clutter on $[6]$ with circuits 
\begin{equation*}
    E(\overline{\C}) = \set{6543, 6542, 6532, 6432, 5432, 5431, 5421}, 
\end{equation*}  
which is underclosed under the reverse ordering on $[6]$: $6<_{r}5<_{r}4<_{r}3<_{r}2<_{r}1$. 

The $1$-layer of $\C$ is
\begin{equation*}
1\text{-layer} = \set{234, 235, 236, 246, 256, 346, 356, 456}
\end{equation*}
and the $i$-layer is empty for $i=2,\dots,6.$ Notice that the complement of the $1$-layer above is $\set{543,542}$, which is underclosed on the order $6<_{r}5<_{r}4<_{r}3<_{r}2.$ To verify that the $1$-layer is cointerval, we consider its layers:
\begin{align*}
    2\text{-layer} = \set{34, 35, 36, 46, 56}, \; \; 3\text{-layer} = \set{46, 56}, \; \; 4\text{-layer} = \set{56}
\end{align*}
Note that the $5$- and $6$-layers are empty. One can see that above $i$-layers are themselves cointerval, and we see that $4\text{-layer} \subseteq 3\text{-layer} \subseteq 2\text{-layer}$ in the $1\text{-layer}$ of $\C$. Thus $\C$ with this ordering is cointerval.

\end{example}

\section{Chordal clutters}\label{sec:chordal}

We next compare our clutters with a notion of chordality introduced by Woodroofe in \cite{Russ09}. We begin with some relevant definitions.

Given a clutter ${\mathcal C}$ and a vertex $v \in V({\mathcal C})$, there are two new clutters that we can form that will be relevant for our study. The \emph{deletion} ${\mathcal C} \backslash v$ is the clutter on vertex set $V({\mathcal C}) \setminus v$ with circuits 
\[\{e : e \in E({\mathcal C}), v \notin e\}.\] 
Notice that ${\mathcal C} \backslash v$ simply removes all circuits that contain $v$. The \emph{contraction} ${\mathcal C}/v$ is the clutter on vertex set $V({\mathcal C}) \backslash v$ with circuits given by the \emph{minimal} sets of
\[\{e \setminus \{v\}: e \in E({\mathcal C})\}.\]
In other words, ${\mathcal C}/v$ removes $v$ from every circuit that contains $v$ \emph{and} then removes any circuits that properly contain others to preserve the clutter property. 

Note that deletion in the context of clutters differs from the deletion of a face of a simplicial complex.  For example, if $\Delta$ is a simplicial complex with a single facet (i.e. $\Delta$ is a simplex), then deleting any vertex in this facet would yield a simplicial complex with a single facet of dimension one lower. However, if $\C$ is a clutter with a single circuit, then deleting any of the vertices in this circuit would yield a clutter with no circuits.

One can check \cite{Russ09} that if $v \neq w$ are vertices of ${\mathcal C}$ then we have the equalities:
\[({\mathcal C} \backslash v)\backslash w = ({\mathcal C} \backslash w)\backslash v, \; \;  ({\mathcal C} / v) / w = ({\mathcal C} / w) / v, \; \; ({\mathcal C} \backslash v) / w = ({\mathcal C} / w) \backslash v.\]

\begin{defn}[Clutter minor] \label{defn: minor}
Suppose ${\mathcal C}$ is a clutter. Any clutter ${\mathcal D}$ obtained from ${\mathcal C}$ by a sequence of deletions and contractions is called a \emph{minor} of ${\mathcal C}$.
\end{defn}

Recall that a vertex $v$ of a graph $G$ (a $2$-clutter) is simplicial if the neighborhood of $v$ is a complete subgraph.  We have the following definition for arbitrary clutters.

\begin{defn}[Simplicial  vertex] \label{defn:simplicial}
    Let $\mathcal{C}$ be a clutter. A vertex $v$ of $\mathcal{C}$ is \emph{simplicial} if for every two circuits $e_1$ and $e_2$ of $\mathcal{C}$ that contain $v$, there is a third circuit $e_3$ such that $e_3 \subseteq (e_1 \cup e_2)\setminus \{v\}$.
\end{defn}

\begin{example}\label{ex:contraction}
    Let $\C$ be the clutter on vertex set $[7]$, with circuits 
    $$E(\C) = \set{1234,1235,1236,2345,2346,2347,3456,4567}.$$
    Notice that $\C$ is underclosed with the given vertex order and that  vertex $7$ is simplicial in $\C$. In particular we have $7$ contained (only) in the circuits $e_1 = 2347$ and $e_2 = 4567$, and we see that $e_3 = 2346 \subseteq (e_1 \cup e_2) \setminus \{7\}$. 
    
    Also observe that the deletion $\C \backslash 7$ has circuits
    $$
    E(\C \backslash 7) = \set{1234,1235,1236,2345,2346,3456}
    $$
    and the contraction $\C / 7$ has circuits
    $$
    E({\mathcal C}/7) = \set{1235,1236,234,456}.
    $$

    We immediately see that $\C \backslash 7$ is underclosed under the given order, but one can check that $\C / 7$ is not underclosed under any order on its vertices. To see this, we can use the fact that the vertex labeled $1$ in an underclosed order cannot be in circuits of different sizes by Lemma~\ref{lem:smallest}.
\end{example}

The class of \emph{chordal} graphs can be characterized as those graphs with the property that every induced subgraph has a simplicial vertex. Inspired by this notion, in \cite{Russ09} Woodroofe introduced the following generalization.

\begin{defn}\cite[Definition 4.3]{Russ09}\label{defn:chordal}
A clutter ${\mathcal C}$ is \emph{chordal} if every minor of ${\mathcal C}$ has a simplicial vertex.
\end{defn}

For example, if $M$ is any matroid, the collection ${\mathcal C} = \text{Cir}(M)$ of circuits of $M$ forms a (typically non-uniform) chordal clutter \cite[Example~4.4]{Russ09}.

We next wish to show that an underclosed clutter is chordal in this sense. Unfortunately, it is not true that any minor of an underclosed clutter is again underclosed, as we see in Example \ref{ex:contraction}. 
We instead show that all minors admit a simplicial vertex.

 Note that if we have an ordering of the vertex set $V$ of a clutter ${\mathcal C}$, then deleting a vertex $v$ induces an ordering on $V \setminus v$ which we call the \emph{compressed labeling} as in \cite{BenSecVar}. Generalizing \cite[Lemma 58]{BenSecVar} to non-uniform clutters we have the following.

\begin{lem}
    If ${\mathcal C}$ is an underclosed clutter, any deletion of ${\mathcal C}$ is also underclosed under the compressed labeling. 
\end{lem}

\begin{proof}
Suppose $\C$ is underclosed under a given ordering of the vertex set $V$, and let $v$ be any vertex. Then ${\mathcal C} \backslash v$ has circuits $\{e : e \in E({\mathcal C}), v \notin e\}$. Consider $f=a_0 a_1 \dots a_k$ a circuit of ${\mathcal C} \backslash v$. We need to show that $g=a_0b_1\dots b_k $ with $ b_i \neq v$ and $b_i\leq a_i$ for $i\in \{1,\dots , k\}$ is also a circuit in ${\mathcal C} \backslash v$.

   Since $v\notin g$, it suffices to show that $g\in E({\mathcal C})$. But we have $f\in E({\mathcal C})$, which allows us to use the underclosed property ${\mathcal C}$ applied to $f$. Hence $g\in E({\mathcal C})$ as desired.
\end{proof}

In particular, deleting vertices preserves the underclosed property.  
However, Example~\ref{ex:contraction} shows that contracting a vertex of an underclosed clutter can possibly take us out of this class.

In what follows, whenever ${\mathcal C}$ is a minor of a clutter ${\mathcal D}$, we assume that all vertex deletions are performed first, followed by vertex contractions. This is justified by the remarks before Definition \ref{defn: minor}, since any order of operations yields the same clutter.

\begin{lem}\label{lem:contract}
Suppose ${\mathcal C}$ is a minor of a clutter ${\mathcal D}$, where $U = \{v_1, v_2, \dots, v_k\}$ are the contracted vertices. Then for any circuit $e \in E({\mathcal C})$, there exists a circuit $f \in E({\mathcal D})$ satisfying $e \subseteq f$ and $f \setminus e \subseteq U$.
\end{lem}

\begin{proof}
    First observe that if $\C$ is obtained from ${\mathcal D}$ by deleting a single vertex, then any circuit $e \in E(\C)$ is also a circuit in ${\mathcal D}$. Hence we can assume that  ${\mathcal C}$ has been obtained from ${\mathcal D}$ by only performing vertex contractions.
    
    Next observe that for any clutter ${\mathcal B}$, the circuits of the contraction ${\mathcal B} / v$ are of the form $e \setminus \set{v}$ for some $e \in E({\mathcal B})$. Hence the result follows by induction on $k$, the number of vertices that are contracted. Indeed, if $k=0$ the claim is clear, and for $k \geq 1$ we note that ${\mathcal C}$ is obtained by contracting the vertex $v_k$ in ${\mathcal C}^\prime = {\mathcal D}/\{v_1, \dots, v_{k-1}\}$. Hence there exists a circuit $e^\prime \in {\mathcal C}^\prime$ with $e \subseteq e^\prime$ and $e^\prime \setminus e \subseteq \{v_k\}$ (note that we could have $e^\prime = e$). By induction we have a circuit $f \in {\mathcal D}$ such that $e^\prime \subseteq f$ and $f \setminus e^\prime \subseteq \{v_1, v_2, \dots, v_{k-1}\}$. The claim follows.
\end{proof}

Although contractions of underclosed clutters are no longer underclosed in general, we are able to find a simplicial vertex in any minor.

\begin{lem} \label{lem:underclosedchordal}
Suppose $\mathcal {C}$ is a clutter obtained as a minor from a clutter ${\mathcal D}$, where ${\D}$ is underclosed according to some ordering $\{1,2, \dots, n\}$ of the vertex set. Let $m$ be the largest vertex of $\mathcal {C}$ under this order. Then $m$ is a simplicial vertex of $\mathcal {C}$. 
\end{lem}

\begin{proof}
Suppose $e_1$ and $e_2$ are distinct circuits of $\C$ containing $m$.
 We want to show that there is a circuit of $\C$ contained in $(e_1 \cup e_2) \setminus m$. Without loss of generality assume $\min(e_1) \leq \min(e_2)$, where $\min(e)$ denotes the smallest element of the set $e \subseteq [n]$.

Suppose ${\mathcal C}$ has been obtained from ${\mathcal D}$ by contracting vertices $U = \{v_1, v_2, \dots, v_k\}$.
From Lemma \ref{lem:contract} let $f_1$ be a circuit of ${\mathcal D}$ with $e_1 \subseteq f_1$ and $f_1 \setminus e_1 \subseteq U$.
Let $v \in e_2 \setminus e_1.$ Observe that  $v \in e_2 \setminus f_1$. 
Since $\D$ is underclosed and
\[\min(f_1) \le \min(e_1) < v < m,\]
we know that ${f_1}' := (f_1 \setminus m) \cup v$ is a circuit of $\D.$ Observe that ${f_1}'\setminus (f_1 \setminus e_1) \subseteq (e_1 \cup e_2) \setminus m$. 
Either ${f_1}' \setminus (f_1 \setminus e_1)$ is a circuit of $\C$ or a subset of this set is a circuit of $\C$ (since $\C$ is a clutter). In either case, we have found a circuit $e_3 \subseteq  (e_1 \cup e_2) \setminus m$ which shows that $m$ is a simplicial vertex.
 \end{proof}

\begin{example}
Although the contraction $\C/7$ in Example~\ref{ex:contraction} is not underclosed, we see that the vertex $6$ is simplicial. Indeed, the only circuits containing $6$ are $e_1=1236$ and $e_2=456.$ Then $(e_1 \cup e_2) \setminus \set{6} = 12345$, and we see that $1235$ is a circuit of $\C/7$ contained in this set. 
\end{example}

\begin{thm}\label{thm:underclosedchordal}
Suppose ${\mathcal C}$ is an underclosed (not necessarily uniform) clutter. Then ${\mathcal C}$ is chordal.
\end{thm}

\begin{proof}
    By Lemma~\ref{lem:underclosedchordal}, every minor of an underclosed clutter has a simplicial vertex, which immediately implies the result.
\end{proof}

Theorem \ref{thm:underclosedchordal} tells us that all underclosed clutters are chordal, and a natural question is to address the converse. Note that the contraction $\C / 7$ from Example \ref{ex:contraction} is a clutter that is chordal but not underclosed.
For $k=2$, it is known that there exists uniform $2$-clutters (graphs) that are chordal but not underclosed (interval), see Figure \ref{fig:chordalnotint}. Indeed, for $k=2$ a graph $G$ is interval if and only if $G$ is chordal and co-comparable.  For larger $k$ it is an open question whether one has such a characterization  (see Section \ref{sec:further} for further discussion). 

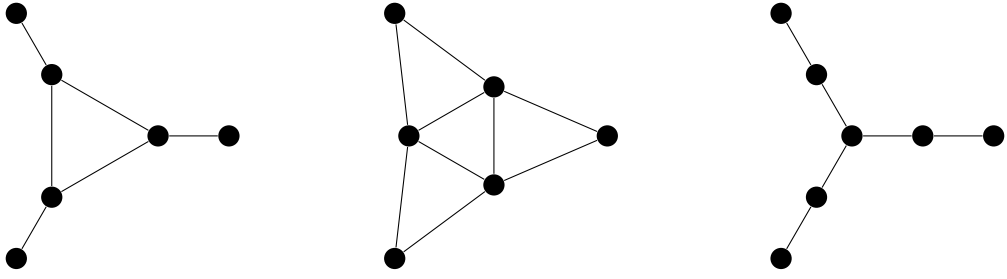
\begin{figure}
     \begin{subfigure}[b]{0.30\textwidth}
          \centering
          
          \begin{tikzpicture}[scale=0.75, every node/.style={fill=black, circle, inner sep=0pt, minimum size=8pt}]

    \node (v1) at (0:1.25cm) {};
    \node (v2) at (120:1.25cm) {};
    \node (v3) at (240:1.25cm) {};

    \node (p1) at (0:2.5cm) {};
    \node (p2) at (120:2.5cm) {};
    \node (p3) at (240:2.5cm) {};

    \draw (v1) -- (v2);
    \draw (v2) -- (v3);
    \draw (v3) -- (v1);

    \draw (v1) -- (p1);
    \draw (v2) -- (p2);
    \draw (v3) -- (p3);

\end{tikzpicture}
          
          \label{fig:A}
     \end{subfigure}
     \begin{subfigure}[b]{0.30\textwidth}
          \centering
          
          \begin{tikzpicture}[scale=0.75, every node/.style={fill=black, circle, inner sep=0pt, minimum size=8pt}]

    \node (v1) at (0:2.5cm) {};
    \node (v2) at (120:2.5cm) {};
    \node (v3) at (240:2.5cm) {};

    \node (p1) at (180:1cm) {};
    \node (p2) at (300:1cm) {};
    \node (p3) at (60:1cm) {};


    \draw (v1) -- (p2);
    \draw (v1) -- (p3);
    \draw (v2) -- (p1);
    \draw (v2) -- (p3);
    \draw (v3) -- (p1);
    \draw (v3) -- (p2);
    \draw (p1) -- (p2);
    \draw (p1) -- (p3);
    \draw (p2) -- (p3);

\end{tikzpicture}
          \label{fig:B}
     \end{subfigure}
     \begin{subfigure}[b]{0.30\textwidth}
          \centering

          \begin{tikzpicture}[scale=0.75, every node/.style={fill=black, circle, inner sep=0pt, minimum size=8pt}]

    \node (v1) at (0:1.25cm) {};
    \node (v2) at (120:1.25cm) {};
    \node (v3) at (240:1.25cm) {};
    \node (v4) at (0:0cm) {};

    \node (p1) at (0:2.5cm) {};
    \node (p2) at (120:2.5cm) {};
    \node (p3) at (240:2.5cm) {};


    \draw (v1) -- (p1);
    \draw (v2) -- (p2);
    \draw (v3) -- (p3);
     \draw (v1) -- (v4);
    \draw (v2) -- (v4);
    \draw (v3) -- (v4);
\end{tikzpicture}

          \label{fig:C}
     \end{subfigure}
     \caption{$2$-clutters (graphs) that are chordal but not underclosed (interval).}
     \label{fig:chordalnotint}
 \end{figure}

In this direction one can use a coning procedure to produce an infinite family of clutters that are chordal and not underclosed. In what follows, if $\C$ is a clutter on vertex set $V$ we let $\D = \C * v$ denote the clutter obtained by adding a disjoint vertex $v$ to $V$, and defining circuits of $\D$ to consist of all sets $e \cup \{v\}$, where $e \in E(\C)$. In this case we say that $\mathcal{D}$ is a \emph{cone} and $v$ is a \emph{cone vertex}.

\begin{prop} \label{lem:coning} Suppose ${\mathcal C}$ is a clutter that is chordal and not underclosed. Then the cone ${\mathcal C} * v$ is also chordal and not underclosed.
\end{prop}

\begin{proof}
Suppose $\C$ is a clutter that is chordal and not underclosed, and let $\D = \C * v$ denote the coned clutter.
 Notice that if $w$ is a simplicial vertex of $\C$, then it is clear that $w$ is a simplicial vertex of $\D$. Therefore, if $\C$ has a simplicial vertex, then so does $\D$. 

Next note that  $\D \backslash v = \varnothing$ and $D / v = \C$, so these minors are chordal. For every vertex $w \in V(\C)$, both the deletion and contraction operations commute with coning with $v$. It follows that every minor of $\D$ is either a minor of $\C$, or a cone over a minor of $\C$. We can conclude that if $\C$ is chordal, then so is $\D$. 

Next note that contracting the cone vertex $v \in \D$ results in a clutter $\C = \D / v$ with circuits $e' \in E(\D / v)$ if and only if $e' \cup \set{v} \in E(\D).$ Thus if $\D$ is a cone that is underclosed under some ordering of the vertex set, then $\C = \D / v$ is also underclosed under the compressed order. The result follows. 
\end{proof}

Using the examples of graphs that are chordal and not interval, Lemma \ref{lem:coning} gives us a natural family of $k$-clutters that are chordal but not underclosed, for any $k \ge 2$.

\begin{rem}
    Note that coning does \emph{not} preserve the property of being underclosed. For instance if one takes $G$ to be the graph consisting of three disjoint edges, then the cone $G * v$ is a $3$-uniform clutter that is not underclosed (see \cite[Remark 40]{BenSecVar}).
\end{rem}

Having established that underclosed clutters are chordal, results of Woodroofe from \cite{Russ09} then give us the following corollaries.

\begin{cor}
If ${\mathcal C}$ is an underclosed clutter, then the independence complex $I({\mathcal C})$ is shellable.
\end{cor}

For the next corollary, recall from Definition \ref{defn:dual} the notion of dual complex $\dual(\C)$ of a $k$-clutter $\C$. This is the simplicial complex whose Stanley--Reisner ideal is the Alexander dual of the circuit ideal $I(\overline{\C})$. In this language we also get the following result from \cite{Russ09}.

\begin{cor}\label{cor:dualisVD}
Suppose ${\mathcal C}$ is an underclosed $k$-clutter. Then $\dual(\C)$ is vertex decomposable.
\end{cor}

This answers a question from \cite{DocEng}, where the authors ask whether the Alexander duals of cointerval
hypergraphs (which we have seen correspond to $\dual(\C)$ for $\C$ underclosed) are vertex decomposable. From this we conclude that circuit ideals of cointerval $k$-hypergraphs have linear quotients.

\section{Lexicographic shellings}

If $\C$ is a uniform underclosed $k$-clutter, the complex $\dual(\C)$ is vertex decomposable and hence shellable. In this section we show that $\dual(\C)$ is in fact \emph{lexicographically shellable}, for a given ordering of the vertex set (see Theorem \ref{thm:DualIsLexShell}). 
In what follows we assume that $\C$ is an $(n-d)$-uniform clutter on $V=[n]$, which means that $\dual(\C)$ is a $(d-1)$-dimensional complex on $[n]$.

To recall our relevant definitions, first note that any ordering of the vertex set of a simplicial complex $\Delta$ defines an ordering on its facets given by the lexicographical order. We say that $\Delta$ is \emph{lexicographically (lex) shellable} if there exists a vertex ordering for which this induced ordering of the facets defines a shelling of $\Delta$ (see \cite{Lexshelling}). For example, any ordering of the ground set of a matroid $M$ defines a lex shelling of its independence complex $I(M)$, and in fact this property characterizes the class of matroids \cite[Theorem~7.3.4]{BjornerHomology}.  It is known that there exist pure vertex decomposable complexes that are not lexicographically shellable under any order of the vertex set; see \cite{GP24}. Our main result in this section is the following.

\begin{thm}\label{thm:DualIsLexShell}
Suppose $\C$ is a uniform underclosed clutter with respect to the labeling of the vertex set. Then $\dual(\C)$ is lexicographically shellable with respect to this ordering.
\end{thm}

In our proof of Theorem \ref{thm:DualIsLexShell}, we will use the following characterization of when a given ordering of facets $F_1, F_2, \dots , F_m$ of a pure simplicial complex is in fact a shelling. Recall that two facets are adjacent if they differ by exactly one vertex. 

\begin{rem}\label{rem:pure_shell}
Adding a facet $F_i$ constitutes a shelling move if and only if the following condition holds: For every $\sigma \in \ideal{F_i} \cap \langle F_1,\dots,F_{i-1}\rangle$, there exists an $F_j$ with $j <i $ such that both $\sigma \subseteq F_j$ and $F_j$ is adjacent to $F_i$.
\end{rem}

We now introduce some notation that will be used in the proof of Theorem~\ref{thm:DualIsLexShell}. Suppose $\Delta$ is any simplicial complex on vertex set $V = [n]$, and let $F \in \Delta$ be a facet. Define $p = p_F$ and $q = q_F$ with $p < q$ to be the two smallest elements of $V$  that are not contained in $F$.
We let $\init(F)$ denote the set of vertices of $F$ that come before $p_F$ and let $\fin(F)$ denote the set of vertices that come after $q_F$.
\begin{align*}
\init(F) &= \{v \in F: v < p_F\} \\
 \fin(F) &= \{v \in F: v > q_F\}
\end{align*}
Note that $\init(F)$ consists of consecutive integers (as do the elements between $p_F$ and $q_F$), whereas $\fin(F)$ may not have this property. With this notation, we have
\[V \setminus F = p_Fq_Fa_3 \cdots a_k.\] 
See Example \ref{ex:pandq} for an example.

\begin{defn}\label{defn:F(j)}
    Assume $\C$ is a uniform clutter on vertex set $V=[n]$. Let $F$ be a facet of $\dual (\C)$ and $j \in F \setminus \init(F).$ Let $i \in V \setminus F$ be the largest vertex less than $j$. We define $F(j) := \parens{F \setminus \set{j} } \cup \set{i}.$

\end{defn}

We refer to Example \ref{ex:pandq} for an example. Observe in Definition~\ref{defn:F(j)} above that $i$ is guaranteed to exist because we assume $j \notin \init(F),$ which means that $j > p_F$ and thus $F(j)$ is well-defined. The following properties of $F(j)$ follow immediately from the definition.

\begin{lem}\label{lem:F(j)_properties}
Suppose $\C$ is a uniform clutter, and let $F$ be a facet of $\dual(\C)$ with $j \in F \setminus \init(F)$. Then $F(j)$, as defined above, satisfies the following.
\begin{itemize}
        \item $j \notin F(j)$,
        \item $|F(j) \cap F| = d$, so that $F(j)$ and $F$ differ in one vertex, and
        \item $F(j)$ is lex-smaller than $F.$ 
    \end{itemize}
\end{lem}

In our constructions, we will want $F(j)$ to be a facet of $\dual(\C)$. This will not be true in general, but will be guaranteed if $j \in \fin(F)$ and $\C$ is underclosed. In particular, we have the following.

\begin{lem}\label{lem:bumping}
    Assume $\C$ is a uniform underclosed clutter on vertex set $V = [n]$, and let $F$ be a facet of $\dual(\C).$ 
    Then for every $j \in \fin(F)$ we have that $F(j) \in \dual(\C)$.
\end{lem}

\begin{proof}
    Suppose $F=v_1v_2 \dots v_d$ is a facet of $\dual(\C)$, and assume $j \in \fin(F)$. As above, let $i \in V\setminus F$ be the largest vertex less than $j$, and let $F(j) = (F \setminus \{j\}) \cup \{i\}.$

    Now consider the complement $V \setminus F = w_1 \dots w_{n-d}$ and let $k$ be the index such that $w_k=i.$ Thus $V \setminus F = w_1 \dots w_{k-1} i w_{k+1} \dots w_{n-d}$ and $V \setminus F(j) = w_1 \dots w_{k-1} j w_{k+1} \dots w_{n-d}$.
    
    Recall that by definition of $\dual(\C)$, we have that $F \in \dual(\C)$ if and only if $V \setminus F \notin \C.$ Since $V \setminus F \notin \C$ by assumption, we see that $V \setminus F(j) \notin \C$ by applying the underclosed property of $\C$. Thus $F(j) \in \dual(\C).$
\end{proof}

We refer to Example \ref{ex:pandq} for an illustration of Definition \ref{defn:F(j)} and Lemmas \ref{lem:F(j)_properties} and \ref{lem:bumping}.

\begin{example}\label{ex:pandq}
Suppose $\C$ is an underclosed $3$-clutter on vertex set $V = [9]$. If $F = 124579$ is a facet of $\dual(\C)$ then by definition $V \setminus F = 368 \notin \C$. Following the notations above, we have $p_F=3$, $q_F = 6$, $\init(F) = \{1,2\}$, and $\fin(F) = \{7,9\}.$

Since $\C$ is underclosed and $368 \notin \C$, we see that $V \setminus F(7) = 378$ and $V \setminus F(9) = 369$ are not in $\C$. This in turn implies that $F(7) = 124569$ and $F(9) = 124578$ are facets of $\dual(\C)$ that are lex-smaller than $F$, as desired. Furthermore, note that $|F(7) \cap F| = |F(9) \cap F| = 5$.

Lastly, for example, consider $F(5) = 123479$, which is a well-defined set because $5 \in F \setminus \init (F)$. However, we cannot apply the underclosed property to show that $V \setminus F(5) = 568 \notin \C$, so $F(5)$ is not necessarily a facet of $\dual(\C)$.
\end{example}

We can now provide a proof of our main result in this section.

\begin{proof}[Proof of Theorem \ref{thm:DualIsLexShell}]
We prove that $\dual(\C)$ is lex shellable by induction on the number of facets. If $\dual(\C)$ consists of a single facet, the result follows.

Now assume $\dual(\C)$ has more than one facet. Consider a facet $F= v_1 v_2 \dots v_d$ of $\dual(\C)$, and assume that 
$$
\dual(\C)_{<F} = \ideal{ G : G ~\text{is a facet of}~ \dual(\C) ~\text{and}~ G <_{lex} F }
$$
is nonempty and shellable under the lexicographic order. We want to show that adding $F$ also constitutes a shelling move. For this we check the condition spelled out in Remark \ref{rem:pure_shell}. Our analysis will depend on 
the structures of $\init(F)$ and $\fin(F)$ as defined above.

For the remainder of the proof, let $\sigma \in \ideal{F} \cap \dual(\C)_{<F}$ and let $G$ be a facet of $\dual(\C)_{<F}$ that contains $\sigma.$ Let 
\begin{align*}
    V \setminus F &= p_F q_F a_3 \dots a_{n-d} \; \; ~ \text{and} \\
    V \setminus G &= p_G q_G b_3 \dots b_{n-d}.
\end{align*}
We will show that $\sigma$ is contained in facet of $\dual(\C)_{<F}$ that is adjacent to $F$.

First assume that $\fin(F) \not \subseteq \sigma$, i.e. there is at least one $j \in \fin(F)$ such that $j \notin \sigma.$ Then Lemma~\ref{lem:bumping} provides an adjacent facet $F(j)$ such that $F(j)$ is lex-smaller than $F$ and $\sigma \subseteq F(j).$ 

Consider the case where $\fin(F) \subseteq \sigma$. In this case we will show that $F(p_G)$ is well-defined (meaning that $p_G \in F$ and $p_F < p_G$), that $V \setminus F(p_G)$ has the form
\begin{align}\label{eq:F(p_G)}
    V \setminus F(p_G) = p_G q_F a_3 \dots a_{n-d}
\end{align}
(meaning additionally that $p_G < q_F$), and finally that $F(p_G) \in \dual(\C)$. This will immediately imply that adding $F$ is a shelling move (via Remark~\ref{rem:pure_shell}), which will complete the proof.

We will establish the first two claims from the previous paragraph by showing $p_F < p_G < q_F$. Since $G$ is lex smaller than $F$, we have $p_F \leq p_G$. Thus if $p_F \notin G$, we must have $p_F = p_G$. In this case we then have that $q_F \leq q_G$, and thus $F \setminus \fin(F) \subseteq G$. But since $\fin(F) \subseteq \sigma \subseteq G$, we have that $F = G$, a contradiction. From this we conclude that $p_F < p_G$.

Assume there are $k$ vertices in $V \setminus F$ that are greater than $q_F.$ Recalling that $\fin(F) \subseteq \sigma \subseteq G,$ this implies that there are at most $k$ vertices in $V \setminus G$ that are greater than $q_F.$ Since $\abs{F}=\abs{G}$, this means $G$ is missing at least two vertices that are less than or equal to $q_F$. Therefore $q_G \le q_F$ and thus $p_G < q_F$. Since we have established that $p_F < p_G < q_F$, we now see that $F(p_G)$ is well-defined and that $V \setminus F(p_G) = p_G q_F a_3 \dots a_{n-d}$ claimed in \eqref{eq:F(p_G)} above. 

To complete the proof, we will now show that $F(p_G) \in \dual(\C)$. Observe that we have already established that $q_G \le q_F$. We will now show that $b_i \le a_i$ for all $i=3,\dots,n-d$ by a similar argument. Since $\fin(F) \subseteq \sigma$ and $\sigma \subseteq G$, we see that if $v > q_F$ and $v \ne a_j$ for any $j$, then $v \ne b_i$ for any $i$. 
Equivalently, if $b_i$ is among $V \setminus G = p_G q_G b_3 \dots b_{n-d}$ and $b_i > q_F$, then $b_i$ also appears in $V \setminus F = p_F q_F a_3 \dots a_{n-d}.$

Assume there exists some $i \geq 3$ such that $b_i > a_i$ (so that $b_i > q_F$). If $b_i \ne a_j$ for any $j$, then we contradict the above observation. If instead $b_i = a_j$ for some $j > i$, we consider the number of vertices after these. There will be $(n-d)-i$ vertices after $b_i$ but only $(n-d)-j$ ``allowable'' vertices from $V \setminus F.$ This is also a contradiction. Therefore $b_i \le a_i$ for all $i$.

Finally, recall that $V \setminus G = p_G q_G b_3 \dots b_{n-d} \notin \C$ by assumption. Since $\C$ is underclosed, the inequalities from above imply that $V \setminus F(p_G) = p_G q_F a_3 \dots a_{n-d} \notin \C$, and thus $F(p_G) \in \dual(\C).$ This completes the proof.
\end{proof}

The following example illustrates the above proof.
 
\begin{example}\label{ex:Illustrate5.1}
Suppose that ${\mathcal C}$ is an underclosed $4$-clutter on vertex set $V = [9]$ and that we wish to add the facet $F = 23478$ to $\dual(\C)_{<F}.$ 
Here we have 
$$
V \setminus F = 1569
$$
and $p_F = 1$, $q_F = 5$, $\init(F) = \varnothing$, and $\fin(F) = \{7,8\}$.

Consider some $\sigma \in \ideal{F} \cap \dual(\C)_{<F}$. If $\fin(F) \not\subseteq \sigma$, then Lemma~\ref{lem:bumping} gives a facet of $\dual(\C)_{<F}$ that is adjacent to $F$ and contains $\sigma$.

If instead $\fin(F) \subseteq \sigma$, there exists a facet $G$ in $\dual(\C)_{<F}$ that contains $\sigma$ and therefore contains $\fin(F) = \{7,8\}$. In the proof of Theorem~\ref{thm:DualIsLexShell}, we see that $p_F < p_G < q_F$, i.e., $1 < p_G < 5.$  
Thus we get $\init(G)$ is either $ \{1\}$ or $\{1,2\}$ or $\{1,2,3\}$, which correspond to $p_G=2,3,4$ respectively.

Note that the lex smallest $5$-sets containing these candidates for $\init(G)$ and containing $\{7,8\}$ are $F(2) = 13478, F(3) = 12478, F(4) = 12378$, respectively. In the proof of Theorem \ref{thm:DualIsLexShell}, we see that one of these must be a facet of $\dual(\C)_{<F}$. Note that this facet will be adjacent to $F$ and contains $\sigma$, as desired.

For example, suppose $G = 12678$, which means that $p_G=3$ and that $V \setminus G = 3459 \notin {\mathcal C}$. Since ${\mathcal C}$ is underclosed, this implies that $V \setminus F(3) = 3569 \notin {\mathcal C}$, which in turn implies that $F(3) = 12478 \in \dual(\C)_{<F}.$
\end{example}

Recall that a shellable $d$-dimensional simplicial complex is homotopy equivalent to a wedge of $d$-spheres. The number of such spheres corresponds to the number of \emph{homology facets} (by definition a facet $F$ in the shelling order whose restriction set is $F$ itself, see \cite[Section 4]{BjoWac}). A consequence of our proof of Theorem \ref{thm:DualIsLexShell} is that one can detect certain sets of these facets.

\begin{rem}\label{rem:homology} 
In the proof of Theorem \ref{thm:DualIsLexShell}, if $F=v_1v_2 \dots v_{d}$ is such that $v_1 \ge 3$, then all vertices of $F$ are in $\fin(F).$ Lemma \ref{lem:bumping} then implies that the boundary of $F$ is already present in $\dual(\C)_{<F}$. 
Thus in our lex shelling of $\dual(\C)$, any such facet $F$ is a \emph{homology facet}. On the other hand, homology facets in this shelling can, in general, have $v_1 < 3$ (see Example \ref{ex:complete_shelling}).
\end{rem}

We end this section with an underclosed complex $\C$ and a lex shelling of its dual.

\begin{example}\label{ex:complete_shelling}
    Let $\C$ be the $3$-clutter on vertex set $V=[6]$ with circuits
    $$
    E(\C) = \set{123,124,125,134,135,234,235,236,245,345,346}.
    $$
    We see that $\C$ is underclosed with the standard order on $[6].$ There are nine $3$-element sets of $V$ that are not circuits of $\C$, and the complements of these are the facets of $\dual(\C).$ Therefore we get
    $$
    \dual(\C) = \ideal{123,124,134,135,234,235,236,245,345}.
    $$
    The above order on the facets of $\dual(\C)$ is a lex shelling. We see that both $234$ and $345$ are homology facets of $\dual(\C)$, so $\dual(\C)$ is homotopy equivalent to a pair of $2$-spheres. 
\end{example}

\section{Further questions}\label{sec:further}

We end with some open questions and suggestions for further avenues of research. In Theorem \ref{thm:underclosedchordal} we prove that underclosed clutters form a subclass of the chordal clutters introduced by Woodroofe, and a natural question is to understand the gap in these two classes.

For instance, recall that if $M$ is matroid its collection of circuits ${\mathcal C}(M)$ forms a chordal clutter. Our first question is then the following.

\begin{question}
For which matroids $M$ is the clutter ${\mathcal C}(M)$ underclosed?
\end{question}

We next address the question of characterizing which chordal clutters are underclosed. The hope is to extend the $d=2$ case, where there is a well-known description of which interval graphs are chordal.
For this, recall that a graph $G$ is a \emph{comparability graph} if it is the underlying graph of some partial order $P$; in particular adjacent vertices are the comparable elements in $P$.
By definition $G$ is a \emph{cocomparability} graph (also called \emph{weakly closed} graph in \cite{BenSecVar}) if it is the complement of a comparability graph.  We then have the following result.

\begin{thm}\cite{GilHoff-Comp&Int}
A graph $G$ is interval if and only if $G$ is chordal and is a cocomparability graph.
\end{thm}

Cocomparability graphs also have algebraic applications to binomial edge ideals, see for instance in \cite{Seccia23}.
It would be interesting to generalize the above result to clutters with larger circuits. Hence we ask the following.

\begin{question}\label{ques:intersection}
 Can we recover the class of underclosed clutters as the intersection of chordal clutters with some other class? 
\end{question}

One approach to answer this would be to make use of other characterizations of these classes in terms of the existence of orderings of the vertex set with certain properties. In particular, the class of cocomparability graphs are those graphs that admit an ordering of the vertex set $V$ that satisfies: for all $a<b<c$ we have
\begin{equation*}
ac \in G \implies ab\in G \text{ or } bc \in G.
\end{equation*}

In \cite{BenSecVar} the authors generalize the above idea to define \emph{weakly closed} clutters. One can see that underclosed clutters are weakly closed, and also weakly closed $2$-clutters are the same as cocomparability graphs. We then ask the following particular version of Question \ref{ques:intersection}.

\begin{question}
    Suppose $\C$ is a clutter that is chordal and weakly closed. Is it true that $\C$ is underclosed? 
\end{question}

Our next question involves the complexity of recognizing underclosed clutters.  Again our motivation comes from graph theory: although checking the interval property for graph naively requires considering all possible orderings of the vertex set, there is in fact a recognition algorithm \cite{Testinginterval} that runs in linear time (in $|V| + |E|$). We ask if there is a similar algorithm for underclosed clutters.

\begin{question}
What can be said about the computational complexity of detecting the underclosed property for a given clutter?
\end{question}

\subsection{Other classes of hypergraphs}

As mentioned in the introduction, many classical classes of graphs have been generalized to the (typically uniform) hypergraph setting with the hope of extending various combinatorial and algebraic applications. This includes the chordal hypergraphs discussed above, and here we mention some others that seem relevant to this general philosophy.

In \cite{Perfect}, Chudnovsky and Kalai studied a notion of perfect hypergraphs, extending the notion from the setting of graphs to $k$-uniform hypergraphs. Recall that a graph $G$ is \emph{perfect} if every subgraph of $G$ (including $G$ itself) has the property that the chromatic number equals the size of the maximum clique.

In \cite{Threshold}, Reiterman, R\"odl, \v Si\v najov\'a, and T\r{u}ma compare various notions of threshold hypergraphs introduced by Golumbic in \cite{Golumbic}.
The class of threshold graphs have many characterizations. For instance, it is known that a graph $G$ is threshold if and only if there exists an ordering of its vertex set such that whenever $ab$ is an edge, we have that $cd$ is an edge for all $c \leq a$ and $d \leq b$ (recall our convention that edges are listed as increasing pairs). From this definition it is clear that threshold graphs are underclosed (interval).

In the classical case, it is known that interval graphs are chordal and hence perfect. It would be interesting to see how underclosed clutters relate to these other classes of hypergraphs discussed above. We leave this for future work.

\begin{question}
How do underclosed clutters relate to the various notions of threshold and perfect hypergraphs from the literature?
\end{question}

\section*{Acknowledgements}

We thank Bruno Benedetti for valuable discussions. Dochtermann was partially supported by Simons Foundation Grant $\#964659$.

\bibliographystyle{amsplain}
\bibliography{references}

\end{document}